\documentclass[11pt,a4paper]{article}
\usepackage{fancyhdr}
\usepackage[british]{babel}

\usepackage{amssymb, amsmath, amsthm, amsfonts, enumerate,bm}
\usepackage{amsmath,setspace,scalefnt}
\usepackage[usenames,dvipsnames,svgnames,table]{xcolor}
\usepackage{graphicx,tikz,caption,subcaption}
\usetikzlibrary{patterns}
\usetikzlibrary{shapes}
\usepackage{fullpage}
\usepackage{hyperref}
\usepackage{enumitem,verbatim}
\usepackage{multicol}
\usepackage{float}
\usepackage{cleveref}

\usepackage[margin=2.5cm]{geometry}

\definecolor{gray}{rgb}{0.25, 0.25, 0.25}

\usepackage{booktabs}  

\counterwithin{figure}{section}

\newtheorem{theorem}{Theorem}[section]
\newtheorem{lemma}[theorem]{Lemma}

\newtheorem*{observation*}{Observation}

\newtheorem*{question*}{Question}
\newtheorem{innerdefinition}{Definition}
\newenvironment{definition*}
  {
   \innerdefinition}
  {\endinnerdefinition}

\theoremstyle{definition}

\theoremstyle{remark}

\newcommand{\ex}{\textrm{ex}}
\newcommand{\EX}{\textrm{EX}}
\newcommand{\spex}{\textrm{spex}}

\usepackage{todonotes}

\title{Spectral extremal problem of the $p$th power of cycles}
\author{{Xinhui Duan,  Lu Lu\footnote{Corresponding author.}\setcounter{footnote}{-1}\footnote{E-mail address:duanxinhui01@163.com (X. Duan); lulugdmath@163.com (L. Lu).}}\\[2mm]
\small School of Mathematics and Statistics, Central South University,\\
\small Changsha, Hunan, 410083, China\\
}

\date{}
\begin{document}
\maketitle

\begin{abstract}
For a cycle $C_k$ on $k$ vertices, its $p$-th power, denoted $C_k^p$, is the graph obtained by adding edges between all pairs of vertices at distance at most $p$ in $C_k$. Let $\ex(n, F)$ and $\spex(n, F)$ denote the maximum possible number of edges and the maximum possible spectral radius, respectively, among all $n$-vertex $F$-free graphs. In this paper, we determine precisely the unique extremal graph achieving $\ex(n, C_k^p)$ and $\spex(n, C_k^p)$ for sufficiently large $n$.\\

\end{abstract}

\section{Introduction}
In this paper, all graphs considered are undirected, finite and simple. Let $G$ be a graph with vertex set $V(G)$ and edge set $E(G)$. Let $K_n$ be the {\it complete graph} on $n$ vertices, and $K_{s,t}$ be the {\it complete bipartite graph} with parts of sizes $s$ and $t$. We write $C_n$ and $P_n$ for the {\it cycle} and {\it path} on $n$ vertices, respectively. The {\it complement} $\overline{G}$ of $G$ is the graph with vertex set $V(G)$ and edge set $\{uv\mid uv\notin E(G)\}$. For two vertex-disjoint graphs $G$ and $H$, the {\it union} of $G$ and $H$ is the graph $G\cup H$ with vertex set $V(G)\cup V(H)$ and edge set $E(G)\cup E(H)$. In particular, we write $kG$ to denote the vertex-disjoint union of $k$ copies of $G$. The {\it join} of $G$ and $H$, denoted by $G \vee H$, is the graph obtained from $G \cup H$ by adding edges between every vertex of $G$ and every vertex of $H$. For positive integers $n_1,\cdots, n_r$, the complete $r$-partite graph $K_{n_1,\cdots,n_r}$ is defined as $\overline{K_{n_1}\cup \cdots \cup K_{n_r}}$. The $r$-partite Tur\'{a}n graph on $n$ vertices, $T_{n,r}$, is the completer-partite graph $K_{n_1,\cdots,n_r}$ with $\max\{|n_i-n_j| \mid 1\leq i,j\leq r\}\leq 1$. Denote by $\nu(G)$ the matching number, $\chi(G)$ the chromatic number and $\delta(G)$ the minimum degree. Let $G^k$ be the {\it $k$-power} of a graph $G$, which is obtained by joining all pairs of vertices with distance at most $k$ in $G$. We write $C_{k}^p$ and $P_{k}^p$ for the $p$-power of a cycle $C_k$ and a path $P_k$, respectively.

A graph $G$ is $F$-free if it contains no subgraph isomorphic to $F$. For a family of graphs $\mathcal{F}$, $G$ is $\mathcal{F}$-free if it is $F$-free for every $F \in \mathcal{F}$. The {\it Tur\'an number} $\ex(n,F)$ is defined as
 \[\textrm{ex}(n, F)=\max\left\{e(G) \mid G \in \mathcal{G}_n \text{ and } G \text{ is } F\text{-free}\right\},\] with $\EX(n, F)$ denoting the set of extremal graphs: 
\[\textrm{EX}(n, F) = \left\{G \in \mathcal{G}_n \mid G \text{ is }F\text{-free and }e(G) = \textrm{ex}(n, F)\right\}.\] 

The study of Tur\'an numbers for $P_k^p$ and $C_k^p$ has a long history in extremal graph theory. Xiao et al. \cite{XKXZ} determined $\ex(n, P_5^2)$ and $\ex(n, P_6^2)$ for all $n$ and conjectured a general formula for $\ex(n, P_k^2)$. This conjecture was later resolved by Yuan \cite{YLT1}, who fully characterized $\EX(n, P_k^p)$. For cycles, Ore \cite{Ore} first determined $\ex(n, C_n)$ in 1961, while Khan and Yuan \cite{KY} established $\ex(n, C_n^2)$ in 2022. The classical Mantel Theorem gives $\ex(n, C_3)$, and F\"uredi and Gunderson \cite{FG} described $\ex(n, C_{2k+1})$ for $k \geq 2$. The most recent results on $\ex(n, C_{2k+2})$ can be found in \cite{H2}. When considering squared cycles ($p=2$), the Tur\'an number $\ex(n, C_k^2)$ is immediate from Tur\'an's Theorem \cite{Turan1941} for $3 \leq k \leq 5$, since $C_k^2 \cong K_k$ in these cases. For $k \geq 6$, Fang and Zhao \cite{FZ} completely determined $\EX(n, C_k^2)$ for sufficiently large $n$ when $3 \nmid k$. Motivated by these works, we aim to fully characterize $\EX(n, C_k^p)$, generalizing Fang and Zhao’s result. Since $C_k^p \cong K_k$ when $p \geq \lfloor k/2 \rfloor$, it suffices to consider the case $p < \lfloor k/2 \rfloor$. Our main result is as follows.

\begin{theorem}\label{thm-1}
For two integers $k$ and $p$ with $k\ge 2p+1$, assume $k=s(p+1)+r$ for some $s\geq1$ and $0\le r<p+1$. Let $0\le m$ and $1\le t\le s$ be the unique integer satisfying $r=ms+t$. If $r\ne 0$, then $\EX(n,C_k^p)=\{ K_{t-1}\vee T_{n-t+1, p+m+1}\}$ for large enough $n$.
\end{theorem}

Given a graph $G$, let $A(G)$ be its adjacency matrix, and $\lambda(G)$ its spectral radius. The {\it spectral Tur\'an number} for a  graph $F$ is defined as
 \[\textrm{spex}(n, F)=\max\left\{\lambda(G) \mid G \in \mathcal{G}_n \text{ and } G \text{ is } F\text{-free}\right\},\]
 with $ \textrm{SPEX}(n, F)$ denoting the set of extremal graphs: 
 \[\textrm{SPEX}(n, F) = \left\{G \in \mathcal{G}_n \mid  G \text{ is } F\text{-free and }\lambda(G) = \textrm{spex}(n, F)\right\}.\] 
 The problem of spectral Tur\'an number has attracted significant attention over the past decades, with results established for numerous specific graphs $F$; comprehensive surveys are found in \cite{CZ,LLF,N1}. Specially, Nikiforov \cite{Nikifrov2008} determined $\spex(n, C_{2k+1})$ for sufficiently large $n$. Nikiforov \cite{Nikiforov2007} and Zhai and Wang \cite{ZW} determined $\textrm{SPEX}(n, C_4)$ for odd $n$ and even $n$, respectively. Subsequently, Zhai and Lin \cite{ZW} determined $\textrm{SPEX}(n, C_6)$. Very recently, Cioab\u{a}, Desai and Tait \cite{CDT} obtained $\textrm{SPEX}(n, C_{2k})$ for $k \geq 3$ and $n$ large enough. Yan, He, Feng and Liu \cite{YHFL} obtained $\textrm{SPEX}(n, C^2_n)$ for $n\geq 15$, and they expected that the extremal graph with respect to $\spex(n, C^k_n)$ for $n$ large enough may be the graph $K_n\setminus E(S_{n-2k+1})$, where $S_{n-2k+1}$ is the star of order $n-2k+1$. For any integer $k\geq 6$ with $3\nmid k$, Fang and Zhao \cite{FZ} also characterized $\textrm{SPEX}(n,C_{k}^2)$ for sufficient $n$. So, we would like to characterize $\textrm{SPEX}(n,C_k^p)$. 
 
\begin{theorem}\label{thm-2}
For two integers $k$ and $p$ with $k\ge 2p+1$, assume $k=s(p+1)+r$ for some $s\geq1$ and $0\le r<p+1$. Let $0\le m$ and $1\le t\le s$ be the unique integer satisfying $r=ms+t$. If $r\ne 0$ and $t\ne s$, then $\textrm{SPEX}(n,C_k^p)=\{ K_{t-1}\vee T_{n-t+1, p+m+1}\}$ for large enough $n$.
\end{theorem}

\section{Proof of Theorem \ref{thm-1}}
For any positive integer $k$, a graph $G$ is called {\it color-$k$-critical} if there exist $k$ suitable edges whose removal decreases its chromatic number, and deleting any $k-1$ vertices does not decrease its chromatic number. It is clear from the definition that any $k$ edges whose removal decreases $\chi(G)$ must form a matching of size $k$. In particular, color-1-critical graphs are just color-critical graphs. In \cite{Sim1974}, Simonovits determined the unique extremal graph for every color-$k$-critical graph.

\begin{lemma}[\cite{Sim1974}]\label{lem-1}
     Let $k\ge 1$ and let $H$ be a a color-$k$-critical graph with chromatic number $ \chi(H) = r + 1 $. Then $\EX(n,H)=\{K_{k-1} \vee T_{n-k+1,r}\}$ for sufficiently large $n$. 
\end{lemma}

\begin{lemma}\label{lem-2}
For two integers $k$ and $p$ with $k\ge 2p+1$, assume $k=s(p+1)+r$ for some $s\geq1$ and $0\le r<p+1$. If $r=0$, then $\chi(C_{k}^p)=p+1$. If $r\ne 0$, let $r=ms+t$ for $m\ge 0$ and $1\leq t\leq s$, then $\chi(C_{k}^p)=p+m+2$ and $C_k^p$ is a color-$t$-critical.
\end{lemma}
\begin{proof}
Assume $V(C_k^p)=\{v_0,v_1,\ldots,v_{k-1}\}$ such that $v_0v_1\cdots v_{k-1}v_0$ forms the cycle $C_k$. First, we consider the case $r=0$.  Color the vertex $v_i$ with the color $a$ if $i\equiv a\pmod p+1$. Let $V_a=\{v_i\mid i\equiv a\pmod {p+1}\}$ be the set of vertices with color $a$. Clearly, $V_a$ is an independent set. Thus, $\chi(C_k^p)\le p+1$. On the other hand, if we color $C_k^{p}$ by $\{0,1,\ldots,p-1\}$, then let $U_a$ be the set of vertices with color $a$. Since $k=sp+s$, there exists $U_b$ with $|U_b|\ge s+1$. We claim that $U_b$ cannot be an independent set. Since otherwise, any pair of vertices in $U_b$ has distance at least $p+1$ in $C_k$. Thus, there would be at least $(s+1)(p+1)$ vertices in $C_k$, which is impossible. Hence, $\chi(C_k^p)\ge p+1$. This yields $\chi(C_k^p)=p+1$.

Next, we consider the case $r\ne 0$. Since $k=s(p+1)+r=s(p+1)+ms+t=t(p+m+2)+(s-t)(p+m+1)$, we construct a coloring $f$: $V(C_k^p)\rightarrow \{0,\ldots,p+m+1\}$ by 
\[f(v_i)=\left\{\begin{array}{ll}
a,&\text{ if }0\le i\le  t(p+m+2)-1 \text{ and }a\equiv i\pmod{p+m+2},\\[2mm]
b,&\text{ if }t(p+m+2)\le i\le k-1\text{ and }b\equiv i-t(p+m+2) \pmod{p+m+1}.
\end{array}\right.\]
For $0\le a\le p+m+1$, let $V_a$ be the set of vertices colored $a$. Clearly, $V_a$ is an independent set. Therefore, $\chi(C_k^p)\le p+m+2$. If $C_k^p$ is colored by $\{0,1,\ldots,p+m+1\}$, let $U_a$ be the set of vertices colored $a$. Since $k=s(p+m+1)+t$ with $1\le t\le s$, there must be a set $U_b$ contains $s+1$ vertices. Similarly, $U_b$ could not be an independent set. Hence, $\chi(C_k^p)\ge p+m+2$. This leads to $\chi(C_k^p)=p+m+2$. 

Finally, we show that the graph $C_k^p$ is a color-$t$-critical when $r\ne 0$. On the one hand, let 
\[B=\{v_{p+m}v_{p+m+1},\ldots,v_{(t-1)(p+m+2)+(p+m)}v_{t(p+m+2)-1}\}\subseteq E(C_k^p).\]
We claim that $\chi(C_k^p-B)\leq p+m+1<\chi(C_k^p)=p+m+2$. Define a coloring $g$: $V(C_k^p-B)\rightarrow \mathbb{Z}_{p+m+1}$ by setting
\[g(v_i)=\left\{\begin{array}{ll}
p+m,& \text{ if } 0\le i\le  t(p+m+2)-1\text{ and }i\equiv p+m+1\pmod {p+m+2},\\[2mm]
f(v_i), &\text{ otherwise }.
\end{array}\right.\]
It is easy to verify that $g$ is a proper coloring and thus $\chi(C_k^p-B)\leq p+m+1<\chi(C_k^p)=p+m+2$. On the other hand, for any subset $S\subseteq V(C_k^p)$ with $t - 1$ vertices, we claim that $\chi(C_k^p - S) = \chi(C_k^p)$. Suppose to the contrary that there exists a subset $S\subset C_k^p$ of $t - 1$ vertices such that $\chi(C_k^p - S) < \chi(C_k^p) = p + m + 2$, where $0 < t < s$. Then, if we color the graph $C_k^p-S$ using $p + m + 1$ colors, there must be at least $s + 1$ vertices assigned the same color. Therefore, each pair of these vertices has distance at least $p+1$ in $C_k$. This indicates there are at least $(s+1)(p+1)$ vertices, which is impossible since $(s+1)(p+1)>k$. Therefore, the graph $C_k^p$ is a color-$t$-critical. 
\end{proof}

\begin{proof}[{\bf Proof of Theorem \ref{thm-1}}]
From  Lemma \ref{lem-2}, we know that $\chi(C_k^p)=p+m+2$ and $C_k^p$ is a color-$t$-critical when $r\ne 0$. Thus, Lemma \ref{lem-1} indicates $\EX(n,C_k^p)=\{K_{t-1}\vee T_{n-t+1,p+m+1}\}$ for $r\ne 0$ and large enough $n$.
\end{proof}

\section{Proof of Theorem \ref{thm-2}}
For brevity, we always assume $ n $ to be sufficiently large throughout the paper. Let $ G\in\textrm{SPEX}(n,C_k^p) $ be an extremal graph, and let $ \bm{x} = (x_1, x_2, \ldots, x_n)^T $ be the Perron vector of $ G $, that is, the eigenvector corresponding to $\lambda(G)$ with maximum component being $1$. We will establish a series of lemmas to characterize the structure of $ G $. For convenience, let $p' = p + m + 1 \geq 2$, and let $\alpha$, $\varepsilon$, and $\varepsilon_1$ be sufficiently small positive constants satisfying 
\[\boxed{\alpha \gg \varepsilon_1 \gg \varepsilon.}\]

\begin{lemma}[\cite{DKLNTW}]\label{lem-3-1}
     Let $F$ be a graph with chromatic number $\chi(F) = r+1$.  For every $\varepsilon > 0$, there exist $\delta > 0$ and $n_{0}$ such that if $G$ is an $F$-free graph on $n\ge n_0$  vertices with $\lambda ( G) \ge \left ( 1- \frac 1r- \delta \right ) n$, then $G$ can be obtained from $T_{n, r}$ by adding and deleting at most $\varepsilon n^{2}$ edges.
\end{lemma}

Given a graph G, the vertex partition $\Pi: V(G)=V_1\cup V_2\cup\cdots\cup V_k$ is said to be an equitable partition if, for each $u\in V_i,|V_j\cap N(u)|=b_{ij}$ is a constant depending only on $i,j$ $(1\leq i,j\leq k)$. The matrix $B_\Pi=(b_{ij})$ is called the quotient matrix of $G$ with respect to $\Pi$.

\begin{lemma}[\cite{CRS}]\label{lem-3-3}
    Let $\Pi\colon V(G) = V_1\cup V_2\ldots\cup V_k$ be an equitable partition of $G$ with quotient matrix $B_{\Pi}$. Then, the largest eigenvalue of $B_{\Pi}$ is just the spectral radius of $G$.
\end{lemma}

\begin{lemma}\label{lem-3-4}
    $\lambda(G)\geq \lambda(K_{t-1}\vee T_{n-t+1,p'})\geq (1-1/p')n.$
\end{lemma}
\begin{proof}
Clearly, $K_{t-1}\vee T_{n-t+1,p'}$ is $C_k^p$-free. Since $e(T_{n-t+1,p'})\geq (1-\frac{1}{p'})\frac{(n-t+1)^2}{2}-\frac{p'}{8}$, we have 
$$
e(K_{t-1}\vee T_{n-t+1,p'})\ge (t-1)(n-1)+(1-\frac{1}{p'})\frac{(n-t+1)^2}{2}-\frac{p'}{8}.
$$
Using the Rayleigh quotient gives
\begin{align*}
    \lambda(G) &\ge \lambda(K_{t-1}\vee T_{n-t+1,p'})\ge \frac{\mathbf{1}^TA(K_{t-1}\vee T_{n-t+1,p'})\mathbf{1}}{\mathbf{1}^T\mathbf{1}}=\frac{2e(K_{t-1}\vee T_{n-t+1,p'})}{n}\\
    &\ge \frac{2}{n}\left(\frac{n^2}{2}-\frac{n^2}{2p'}+(t-1)\left(\frac{n}{p'}-1\right)+\frac{(t-1)^2}{2}-\frac{(t-1)^2}{2p'}-\frac{p'}{8}\right)\\
    &\ge (1-\frac{1}{p'})n.
    \end{align*}
for sufficient large $n$, as desired.
\end{proof}

Next we show that $G$ contains a large maximum cut.
\begin{lemma}\label{lem-3-5}
  We have $e(G)\geq e(T_{n,p'})-\varepsilon n^2$ . Furthermore, $G$ admits a partition $V(G)=\bigcup_{i=1}^{p'}V_i$ such that $ \sum_{1 \leq i < j \leq p'} e(V_i, V_j)$ attains the maximum, $\sum_{i=1}^{p'} e(V_i) \leq \varepsilon n^2$, and $||V_i|-n/p'|\leq 3\sqrt{\varepsilon} n$ for $1\le i\le p'$.
\end{lemma}
\begin{proof}
    Since $G$ is $C^p_k$-free, by Lemmas \ref{lem-3-1} and \ref{lem-3-4}, for sufficiently large $n$, there exists a partition of $V(G)=U_1\cup U_2 \cdots \cup U_{p'}$ such that $e(G)\geq e(T_{n,p'})-\varepsilon n^2$, $\sum^{p'}_{i=1}e(U_i)\leq \varepsilon n^2$, and $\lfloor\frac{n}{p'}\rfloor\leq |U_i|\leq \lceil\frac{n}{p'}\rceil$ for $1\le i\le p'$. Therefore, $G$ has a partition $V(G)=V_1\cup\cdots\cup V_{p'}$ such that the number of crossing edges of $G$ attains the maximum, and 
    $$\sum^{p'}_{i=1}e(V_i)\leq \sum^{p'}_{i=1}e(U_i)\leq \varepsilon n^2.$$
    Let $a=\max\{||V_j|-\frac{n}{p'}|, 1\le j\le p'\}$. Without loss of generality, we may assume that $||V_j|-\frac{n}{p'}|=a.$ Then
    \begin{align*}
        e(G)&\leq \sum_{1\leq i<j\leq p'}|V_i||V_j|+\sum^{p'}_{i=1}e(V_i)\\
        &\leq |V_1|(n-|V_1|)+\sum_{2\leq i<j\leq p'}|V_i||V_j|+\varepsilon n^2\\
        &=|V_1|(n-|V_1|)+\frac{1}{2}\left(\left(\sum^{p'}_{j=2}|V_j|\right )^2-\sum^{p'}_{j=2}|V_j|^2\right)+\varepsilon n^2\\
        &\leq |V_1|(n-|V_1|)+\frac{1}{2}(n-|V_1|)^2-\frac{1}{2(p'-1)}(n-|V_1|)^2+\varepsilon n^2\\
        &<-\frac{p'}{2(p'-1)}a^2+\frac{p'-1}{2p'}n^2 +\varepsilon n^2,
    \end{align*}
    where the last second inequality holds by H\"older's inequality, and the last inequality holds since $||V_1|-\frac{n}{r}|=a$.  On the other hand, since $e(G)\geq e(T_{n,p'})-\varepsilon n^2$, we have
    $$e(G)\geq e(T_{n,p'})-\varepsilon n^2 \geq \frac{p'-1}{2p'} n^2-\frac{p'}{8}- \varepsilon n^2 \geq \frac{p'-1}{2p'} n^2-2\varepsilon n^2.$$
    Therefore, $\frac{p'}{2(p'-1)}a^2<3\varepsilon n^2$, which implies that $a<\sqrt{\frac{6(p'-1)\varepsilon}{p'}n^2}<3\sqrt{\varepsilon} n$. 
    
    The proof is completed.
\end{proof}

In what follows, we will always assume $V(G)=\sum_{i=1}^{p'} V_i$ is the partition characterized in Lemma \ref{lem-3-5}. Let $L=\{v\in V(G)\mid d(v)\leq (1-\frac{1}{p'}-10\sqrt{\varepsilon})n\}$. 
\begin{lemma}\label{lem-3-6}
$|L|\leq \sqrt{\varepsilon}n$
\end{lemma}
\begin{proof}
    Suppose to the contrary that $|L|>\sqrt{\varepsilon}n$. Then there exists a subset $L'\subseteq L$ with $|L'|=\lfloor\sqrt{\varepsilon}n\rfloor$. Combining these with Lemma \ref{lem-3-5} that $e(G)\geq e(T_{n,p'}) -\varepsilon n^2$, we get
   \begin{align*}
       e(G-L')&\geq e(G)-\sum_{v\in L'}d(v)\\
       &\geq (1-\frac{1}{p'})\frac{n^2}{2}-\varepsilon n^2-\frac{p'}{8}-\sqrt{\varepsilon}n(1-\frac{1}{p'}-10\sqrt{\varepsilon})n\\
       &\ge (1-\frac{1}{p'})\frac{n'^2}{2}+(1-\frac{1}{p'})n\lfloor\sqrt{\varepsilon}n\rfloor-(1-\frac{1}{p'})\frac{(\sqrt{\varepsilon}n)^2}
       {2}-\sqrt{\varepsilon}(1-\frac{1}{p'}+9\sqrt{\varepsilon})n^2-\frac{p'}{8}\\
       &\geq (1-\frac{1}{p'})\frac{n'^2}{2}+8\varepsilon n^2\\
       &\geq (t-1)(n'-1)+(1-\frac{1}{p'})\frac{(n'-t+1)^2}{2}\\
       &>\ex(n',C_k^p),
   \end{align*}
   where $n'=n-\lfloor\sqrt{\varepsilon} n\rfloor$. Since $e(G-L')>\ex(n-|L'|,C_k^p)$, $G-L'$ contains a $C_k^p$ as a subgraph. This contradicts the fact that $G$ is $C_k^p$-free.
\end{proof}

 Let $W=\bigcup_{i=1}^{p'}W_i$, where $W_i=\{v\in V_i\mid d_{V_i}(v)\geq 2\sqrt{\varepsilon}n\}$. 
\begin{lemma}\label{lem-3-7}
  $|W|\leq \sqrt{\varepsilon}n$
\end{lemma}
\begin{proof}
    For $1\le i\le p'$, we have 
    $$
    2e(V_i)=\sum_{v\in V_i}d_{V_i}(v)\geq \sum_{v\in W_i}d_{V_i}(v) \geq 2|W_i|\sqrt{\varepsilon}n.
    $$
    Combining this with Lemma \ref{lem-3-5}, we have 
    $$
    \varepsilon n^2\geq \sum_{i=1}^{p'}e(V_i)\geq |W|\sqrt{\varepsilon}n.
    $$
    It follows that $|W|\leq \sqrt{\varepsilon}n$.
\end{proof}

Let $v^*$ be the vertex satisfying $x_{v^*} = \max \{ x_v \mid v \in V(G) \setminus W \}$, and let $u^*$ be the vertex satisfying $x_{u^*} = \max \{ x_v \mid v \in V(G) \}=1$. 
\begin{lemma}\label{lem-3-8}
    $L=\emptyset$.
\end{lemma}
\begin{proof}
   Recall that $x_{u^*}=\max\left\{x_v\mid v\in V(G)\right\}=1$, then
    $$
    \lambda(G)=\lambda(G)x_{u^*}< |W|+(n-|W|)x_{v^*}.
    $$
    Combing with Lemma \ref{lem-3-4}, we have 
\begin{equation}\label{ineq2}
    x_{v^*}>\frac{\lambda(G)-|W|}{n-|W|}\ge\frac{\lambda(G)-|W|}{n}\ge 1-\frac{1}{p'}-\sqrt{\varepsilon}>1-\frac{5}{4p'}.      
\end{equation}
On the other hand, we have 
\begin{align*}
    \lambda(G) x_{v^*}&=\sum_{vv^*\in E(G)}x_v=\sum_{v\in N(v^*)\cap W}x_v+\sum_{v\in N(v^*)\setminus W}x_v\\[2mm]
    &\le |W|+(d(v^*)-|W|)x_{v^*}.
\end{align*}
Therefore,
\[\begin{array}{lll}
    d(v^*) &\ge& \lambda(G) +|W|-\frac{|W|}{x_{v^*}}\ge \lambda(G) -\frac{5}{4p'-5}|W|\\[2mm]
    &>&(1-\frac{1}{p'})n-\frac{5\sqrt{\varepsilon}}{4p'-5}n>(1-\frac{1}{p'}-10\sqrt{\varepsilon})n.
\end{array}\]
This yields $v^*\in V(G)\setminus (W \cup L)$. Without loss of generality, we assume that $v^*\in V_1\setminus (W\cup L)$. We have
\begin{align*}
    \lambda(G) x_{v^*}
    &=\sum_{u\in N(v^*)\cap (W\cup L)}x_u+\sum_{u\in N(v^*)\cap (V_1\setminus (W\cup L))}x_u+\sum_{u\in N(v^*)\cap (\cup_{i=2}^{p'}V_i\setminus (W\cup L))}x_u\\
    &<|W|+|L|x_{v^*}+2\sqrt{\varepsilon} nx_{v^*}+\sum_{u\in N(v^*)\cap (\cup_{i=2}^{p'}V_i\setminus (W\cup L))}x_u,
\end{align*}
which implies that 
\begin{equation}\label{ineq1}
    \sum_{u\in N(v^*)\cap (\cup_{i=2}^{p'}V_i\setminus (W\cup L))}x_u\ge (\lambda(G)-|L|-2\sqrt{\varepsilon} n)x_{v^*}-|W|.     
\end{equation}

Suppose to the contrary that there is a vertex $v_0\in L$. Assume that $v_0\in L\cap V_j$ for some $j$. Then $d(v_0)\le (1-\frac{1}{p'}-10\sqrt{\varepsilon})n$. Let $G'$ be the graph defined as
\[G'=G-\left\{uv_0\mid uv_0\in E(G)\right\}+\left\{wv_0\mid w\in\cup_{1\le i\le p',i\ne j} V_i \setminus(W\cup L)\right\}.\] Then $G'$ is $C_k^p$-free. Otherwise, assume that $G'$ contains a subgraph $H$ which is isomorphic to $C_k^p$. Then $v_0\in H$. Let $N_H(v_0)=\{v_1,\cdots,v_a\}$, where $2p\le a\le k-1$. For any vertex $v_i\in N_H(v_0)$, since $v_i\notin L$, we have $d_G(v_i)>(1-\frac{1}{p'}-10\sqrt{\varepsilon})n$. Moreover, since $v_i\notin W$ and $|V_{\ell}|\leq (\frac{n}{p'}+3\sqrt{\varepsilon}n)$ for $1\le \ell\le p'$, we have 
    \[
        d_{V_{j}}(v_i)\geq d_G(v_i)-2\sqrt{\varepsilon}n-(p'-2)(n/p'+3\sqrt{\varepsilon}n)\geq (\frac{1}{p'}-3\sqrt{\varepsilon}p'-6\sqrt{\varepsilon})n,
    \]
and thus
\begin{align*}
&|(\cap_{i\in[a]}N_{V_{j}}(v_i))\setminus(W \cup L\cup H)| \\
&\ge  \sum_{i=1}^{a}d_{V_j}(v_i)-(a-1)|V_j|-|W|-|L|-|H|  \\
&> (1/p'-9a\sqrt{\varepsilon}-3ap'\sqrt{\varepsilon})n-k \\
&>1.
\end{align*}
This means that $v_1, \ldots, v_{a}$ have another common neighbor $v$ in $V_j\setminus H$. Therefore, $(V(H)\cup \{v\})\setminus\{v_0\}$ forms a new $C_k^p$ of $G$, a contradiction.

On the other hand, by (\ref{ineq2}) and (\ref{ineq1}), we have 
\begin{align*}
    \lambda(G')-\lambda(G)
    &\ge\frac{{x}^T (A(G')-A(G)){x}}{{x^T}{x}}=\frac{2x_{v_0}}{{x^T}{x}}\left( \sum_{u\in N(v^*)\cap(\cup_{1\le i\le p', i\ne j}V_i)\setminus (W\cup L)}x_w-\sum_{uv_0\in E(G)}x_u\right)\\
    &\ge \frac{2x_{v_0}}{{x^T}{x}}\left(\left(\left(\lambda(G)-|L|-2\sqrt{\varepsilon}n)x_{v^*}-|W|\right)-(|W|+(d(v_0)-|W|)x_{v^*}\right)\right)\\
    &\ge \frac{2x_{v_0}}{{x^T}{x}}((\lambda(G)-|L|-2\sqrt{\varepsilon}n-d(v_0)+|W|)x_{v^*}-2|W|)\\
    &> \frac{2x_{v_0}}{{x^T}{x}}(7\sqrt{\varepsilon} n(1-\frac{5}{4p'})-2\sqrt{\varepsilon}n)\\
    &\ge\frac{2x_{v_0}}{{x^T}{x}}\left((5-\frac{10}{p'})\sqrt{\varepsilon}n\right)\\
    &>0,
\end{align*}
where the last inequality holds by $p'\geq 2$. This contradicts the maximality of $\lambda(G)$.
\end{proof}

\begin{lemma}\label{clm-3-12}
    For any $w\in W_i$, there is at most one index $j \neq i$ such that $d_{V_j}(w)\leq (\frac{1}{2p'}+\alpha)n$. Moreover, for any $j\ne i$, we have $d_{V_j}(w)>(\frac{1}{2p'}-3\sqrt{\varepsilon}p')n$.
\end{lemma}
\begin{proof}
 Conversely, assume that $j_1,j_2\neq i$ such that $d_{V_{j_1}}(w),d_{V_{j_2}}(w)\leq (\frac{1}{2p'}+\alpha)n$. Since $\bigcup_{i=1}^{p'}V_i$ is a maximal cut of $G$, we have $d_{V_i}(w)\leq d_{V_{j_1}}(w)\leq (\frac{1}{2p'}+\alpha)n$. Therefore, 
 \[d(w)\leq 3(\frac{1}{2p'}+\alpha)n+(p'-3)(1/p'+3\sqrt{\varepsilon})n<(1-\frac{1}{p'}-10\sqrt{\varepsilon})n,\] a contradiction.

    For any vertex $w$ in $W_i$, we have $d_{V_i}(w)\leq d_{V_j}(w)$, where $d_{V_j}=\min\{d_{V_{\ell}}(w)\mid \ell\neq i\}$. Then 
    \[
    2d_{V_j}(w)+(p'-2)(\frac{n}{p'}+3\sqrt{\varepsilon}n)\geq d(v)\ge (1-\frac{1}{p'}-10\sqrt{\varepsilon})n,
    \]
    Thus, $d_{V_j}(w_i)>(\frac{1}{2p'}-3\sqrt{\varepsilon}p')n$. 
\end{proof}

\begin{lemma}\label{lem-3-9}
    For any $v\in V(G)\setminus W$, we have $x_v\geq (1-\varepsilon_1)x_{v^*}$.
\end{lemma}
\begin{proof}
Without loss of generality, assume $ v^* \in V_1 \setminus W $. Thus, $d(v^*)\leq (p'-1)(n/p'+3\sqrt{\varepsilon}n)+2\sqrt{\varepsilon}n<(1-1/p'+3\sqrt{\varepsilon}p')n$.

Let us first consider $ v \in V_1 \setminus W $ with $ v \neq v^* $. Since $L=\emptyset$ and $v,v^*\not\in W$, we get 
\[d_{V\setminus V_1}(v),d_{V\setminus V_1}(v^*)\ge (1-\frac{1}{p'}-12\sqrt{\varepsilon})n.\] This leads to
\begin{align*}
    |N(v) \cap N(v^*)|&\geq |N_{V\setminus V_1}(v) \cap N_{V\setminus V_1}(v^*)|\\
    &\geq d_{V\setminus V_1}(v)+d_{V\setminus V_1}(v^*)-|V\setminus V_1|\\
    &\geq 2(1-\frac{1}{p'}-12\sqrt{\varepsilon})n-\left(n-(\frac{n}{p'}-3\sqrt{\varepsilon}n)\right)\\
    &= (1-1/p'-27\sqrt{\varepsilon})n,
\end{align*}
and 
$$|N(v^*)\setminus (N(v)\cup W)|\leq d(v^*)-|W|-|N(v) \cap N(v^*)|\leq (26\sqrt{\varepsilon}+3\sqrt{\varepsilon}p')n.$$
Therefore, we get
\[\begin{array}{lll}
    \lambda (x_v-x_{v^*})
    &=&\left(\sum_{w\in (N(v)\setminus N(v^*))\cap W}x_w+\sum_{w\in (N(v)\setminus N(v^*))\setminus W}x_w\right)\\[2mm]
    &&-\left(\sum_{w\in (N(v^*)\setminus N(v))\cap W}x_w+\sum_{w\in (N(v^*)\setminus N(v))\setminus W}x_w\right)\\[2mm]
    &\geq&-\left(\sum_{w\in (N(v^*)\setminus N(v))\cap W}x_w+\sum_{w\in (N(v^*)\setminus N(v))\setminus W}x_w\right)\\[2mm]
    &\geq&-\sum_{w\in N(v^*)\cap W}x_w-\sum_{w\in N(v^*)\setminus (N(v)\cup W)}x_w\\[2mm]
    &\ge &-|W|-|N(v^*)\setminus (N(v)\cup W)|\\[2mm]
    &\geq&-\sqrt{\varepsilon}n-(26\sqrt{\varepsilon}+3\sqrt{\varepsilon}p')nx_{v^*}\\[2mm]
    &\geq&-(27+3p')\sqrt{\varepsilon}nx_{v^*}\ge -17p'\sqrt{\varepsilon}nx_{v^*}.
\end{array}\]
This implies that $x_v \geq \left(1-\frac{17{p'}^2}{p'-1}\sqrt{\varepsilon}\right)x_{v^*}\geq (1-\varepsilon_1)n$.

Next, we estimate the eigenvector entries for other vertices. Without loss of generality, consider $ v \in V_2 \setminus W $. For $i\neq 2$, we have
\begin{align*}
    d_{V_i}(v)\geq d(v)-d_{V_2}(v)-(p'-2)(n/p'+3\sqrt{\varepsilon}n)\geq (1/p'-6\sqrt{\varepsilon}p')n.
\end{align*}
and $ d_{V_i\setminus W}(v)\geq (1/p'-7\sqrt{\varepsilon}p')n.$
Note that 
\[\sum_{w\in N_{V_1}(v)\setminus W}x_w\geq d_{V_1\setminus W}(v)\left(1-\frac{17{p'}^2}{p'-1}\sqrt{\varepsilon}\right)x_{v^*} \geq (\frac{1}{p'}-24\sqrt{\varepsilon}p')nx_{v^*}.\]
Similarly, noticing that $3x_{v^*}>1$ and $p'\ge 2$, we have
\[\begin{array}{lll}
    \lambda (x_v-x_{v^*})&\ge &\sum_{w\in N_{V_1}(v)\setminus W}x_w-\sum_{w\in N(v^*)\cap W}x_w-\sum_{w\in N(v^*)\setminus (N(v)\cup W)}x_w\\[2mm]
    &\geq& (\frac{1}{p'}-24\sqrt{\varepsilon}p')nx_{v^*}-\sqrt{\varepsilon}n-\sum_{w\in N(v^*)\setminus (N(v)\cup W)}x_w\\[2mm]
    &\geq& (\frac{1}{p'}-24\sqrt{\varepsilon}p')nx_{v^*}-\sqrt{\varepsilon}n-\sum_{w\in N_{V_1}(v^*)\setminus (N_{V_1}(v)\cup W)}x_w\\[2mm]
    &&-\sum_{w\in N_{V_2}(v^*)\setminus (N_{V_2}(v)\cup W)}x_w-\sum_{i=3}^{p'}\sum_{w\in N_{V_i}(v^*)\setminus (N_{V_i}(v)\cup W)}x_w\\[2mm]
    &\ge &(\frac{1}{p'}-24\sqrt{\varepsilon}p')nx_{v^*}-3\sqrt{\varepsilon}nx_{v^*}-2\sqrt{\varepsilon}nx_{v^*}-(1/p'+3\sqrt{\varepsilon})nx_{v^*}\\[2mm]
    &&-(p'-2)\left( n/p'+3\sqrt{\varepsilon}n-(n/p'-6\sqrt{\varepsilon}p'n)\right)x_{v^*}\\[2mm]
    &\ge &-8({p'}^2+p'+1)\sqrt{\varepsilon}nx_{v^*}\ge -16{p'}^2\sqrt{\varepsilon}nx_{v^*}
\end{array}\]
It follows that $x_u\geq (1-\varepsilon_1)x_{v^*}$.

In summary, we have shown that for any $ v \in V(G) \setminus W $, $ x_v \geq (1 - \varepsilon_1)x_{v^*} $. 
\end{proof}

\begin{lemma}\label{lem-3-10}
    $\sum_{i=1}^{p'} \nu(G[V_i \setminus W])\leq t-1$ and $\sum_{i=1}^{p'} e(G[V_i \setminus W])\leq 4(t-1)\sqrt{\varepsilon}n$. 
\end{lemma}
\begin{proof}
First, we consider $\sum_{i=1}^{p'} \nu(G[V_i \setminus W])$.  Suppose to the contrary that $\sum_{i=1}^{p'} \nu(G[V_i \setminus W])\geq t$. We will obtain a contradiction by finding $k$ vertices $\{u_1,u_2,\ldots,u_k\}$ that form a $C_k^p$. Let $M=\{v^{(1)}_{i_1,1}v^{(1)}_{i_1,2},v^{(2)}_{i_2,1}v^{(2)}_{i_2,2},\ldots,v^{(t)}_{i_t,1}v^{(t)}_{i_t,2}\}$ be a matching of $\bigcup_{i=1}^{p'} G[V_i\setminus W]$ with $v^{(j)}_{i_j,1},v^{(j)}_{i_j,2}\in V_{i_j}$. Without loss of generality, assume $i_1=1$ and $i_1\le i_2\le \cdots\le i_t$. For $1\le j\le t$, let 
  \[u_{(p'+1)(j-1)+i_j}=v^{(j)}_{i_j,1}\text{ and }u_{(p'+1)(j-1)+i_j+1}=v^{(j)}_{i_j,2}.\]
 In detail, we get $u_1=v^{(1)}_{i_1,1}$, $u_2=v^{(1)}_{i_1,2}$, $u_{(p'+1)+i_2}=v^{(2)}_{i_2,1}$, $u_{(p'+1)+i_2+1}=v^{(2)}_{i_2,2}$, and so on. For any $3\le \beta\le k$, let $K_{\beta}^-$ be the set of vertices $u_x$ that are already defined with $\beta-p\le x\le \beta-1\pmod k$, and let $K_{\beta}^+$ be the set of vertices $u_x$ that are already defined with $\beta+1\le x\le \beta+p\pmod k$. Denote by $K_{\beta}=K_{\beta}^-\cup K_{\beta}^+$. Moreover, we define $\overline{\beta}$ as follows: 
 \begin{itemize}
 \item If $3\le \beta\le t(p'+1)$, then there exist unique $\ell\ge 0$ and $1\le h\le p'+1$ such that $\beta=\ell(p'+1)+h$. We define $\overline{\beta}=h$ if $1\le h\le i_{\ell+1}$, and $\overline{\beta}=h-1$ if $i_{\ell+1}+1\le h\le p'+1$.
 \item If $t(p'+1)+1\le \beta\le k$, then there exist unique $\ell\ge 0$ and $1\le h\le p'$ such that $\beta-t(p'+1)=\ell p'+h$. We define $\overline{\beta}=h$.
 \end{itemize}
 
 We claim that, if $u_{\beta}$ is defined, then $u_{\beta}\in V_{\overline{\beta}}$. First, the vertices $u_{(j-1)(p'+1)+i_j}$ and $u_{(j-1)(p'+1)+i_j}$ are defined for $1\le j\le t$. Clearly, they satisfy $u_{\beta}\in V_{\overline{\beta}}$. If $u_{\beta}$ is the vertex that is not defined with the minimum index, from the definition of $\overline{\beta}$, we get $v\not\in V_{\overline{\beta}}$ for any $v\in K_{\beta}$. Therefore, for any $v\in  K_{\beta}$, we have
\begin{equation}\label{ineq3}
d_{V_{\overline{\beta}}}(v) \geq d(v)-2\sqrt{\varepsilon}n-(p'-2)(n/p'+3\sqrt{\varepsilon}n)\geq
\left( \frac{1}{p'} - 6\sqrt{\varepsilon}p' \right)n > \left( \frac{1}{p'} - \alpha \right)n.
\end{equation}
This yields that
\[\left|\bigcap_{v\in K_{\beta}}N_{V_{\overline{\beta}}}(v)\right|\ge  |K_{\beta}|\left( \frac{1}{p'} - \beta \right)n - (|K_{\beta}| - 1)\left( \frac{1}{p'} + 3\sqrt{\varepsilon} \right)n \geq \alpha n.\]
 Thus, we could choose $u_{\beta}$ to be an arbitrary vertex in $\left(\bigcap_{v\in K_{\beta}}N_{V_{\overline{\beta}}}(v)\right)\setminus W$, and $u_{\beta}\in V_{\overline{\beta}}$. Hence, the desired vertices $\{u_1,\ldots,u_k\}$ are obtained by repeating these processes.

Next, consider $\sum_{i=1}^{p'} e(G[V_i \setminus W])$. Let $M_i'$ be maximum matchings of $G[V_i\setminus W]$ for $1\le i\le p'$. Therefore, any edge in $G[V_i\setminus W]$ must share a common endpoint with an edge in $M_i'$. Note that, $d_{V_i}(v)\le 2\sqrt{\varepsilon}n$. Therefore, we get $e(G[V_i\setminus W])\le 2|M_i'|\cdot (2\sqrt{\varepsilon}n)=4|M_i'|\sqrt{\varepsilon}n$. Thus, $\sum_{i=1}^{p'} e(G[V_i \setminus W])\le 4\sqrt{\varepsilon}n(\sum_{i=1}^{p'}|M_i'|)\le 4(t-1)\sqrt{\varepsilon}n$.
\end{proof}

\begin{lemma}\label{lem-3-11}
    $t-1\leq|W|\leq p'(p'-1)t$.
\end{lemma}
\begin{proof}
  First, we show $t-1\leq|W|$. Suppose to the contrary that $|W| <t-1$. Select a vertex $v' \in V_1\setminus W$. Define the graph $G'$ as
\[
G' = G - \{ uw \mid uw \in \bigcup_{i=1}^{p'}E( G[V_i \setminus W])\} + \left\{ v'w \mid w \in V(G) \right\}.
\]
Clearly, $G'$ is a subgraph of $K_{t-1} \vee H$, where $H$ is a complete $p'$-partite graph on $n-t+1$ vertices. Consequently, $G'$ does not contain $C_k^p$ as a subgraph. By Lemmas \ref{lem-3-9}, \ref{lem-3-10} and \eqref{ineq2}, we have 
\begin{align*}
    \lambda(G')-\lambda(G)&\geq \frac{2}{xx^T}\left(\sum_{uv\in E(G')}x_ux_v-\sum_{uv\in E(G)}x_ux_v\right)\\
    &\geq \frac{2}{xx^T}\left(\sum_{u\notin N_G(v')}x_ux_{v'}-\sum_{uv\in E(\bigcup_{i=1}^{p'} G[V_i \setminus W])}x_ux_v\right)\\
    &\geq \frac{2}{xx^T}\left(\sum_{u\notin (N_{G}(v')\cap V_1)\setminus W}x_ux_{v'}-\sum_{uv\in \bigcup_{i=1}^{p'} E(G[V_i \setminus W])}x_ux_v\right)\\
    &\geq \frac{2}{xx^T}\left((1/p'-6\sqrt{\varepsilon})(1-\varepsilon_1)^2nx_{v^*}^2-4(t-1)\sqrt{\varepsilon}n\right)\\
    &\geq \frac{2n}{xx^T}\left((1/p'-6\sqrt{\varepsilon})(1-\varepsilon_1)^2(1-\frac{5}{4p'})^2-4(t-1)\sqrt{\varepsilon}\right)\\
    &>0,
\end{align*}
a contradiction. 

Next, we prove that $|W| \leq p'(p'-1)t$. Suppose to the contrary that $|W| > p'(p'-1)t$. By the pigeonhole principle, there is a set, say $W_1$, contains at least $(p'-1)t$ vertices. For $w\in W_1$, denote by $i_w$ the number satisfying $2\le i_w\le p'$ and $d_{V_{i_w}}=\min\{d_{V_j}(w)\mid 2\le j\le p'\}$. Moreover, for $2\le j\le p'$, let $W_{1,j}=\{w\in W_1\mid i_w=j\}$. Since $\sum_{j=2}^{p'}|W_{1,j}|=|W_1|\ge (p'-1)t$, there exists a set, say $W_{1,2}$, containing at least $t$ vertices. Assume $W_1'=\{v^{(1)}_{1,2}, v^{(2)}_{1,2},\ldots, v^{(t)}_{1,2}\}\subseteq W_{1,2}$. Note that $\nu(G[W_1',V_1\setminus W])=t$. Let
 \[M_1=\{v^{(1)}_{1,1}v^{(1)}_{1,2},v^{(2)}_{1,1}v^{(2)}_{1,2},\ldots,v^{(t)}_{1,1}v^{(t)}_{1,2}\}\]  be the matching of size $ t $ in $ G[W_1',V_1\setminus W]$. By the same method as that in the proof of Lemma \ref{lem-3-10}, we could also define the vertices $\{u_1,u_2,\ldots,u_k\}$ that form a $C_k^p$, which leads to the contradiction. In detail, adopting the setting and notation from the proof of Lemma \ref{lem-3-10}, we claim that, if $u_{\beta}$ is defined, then $u_{\beta}\in V_{\overline{\beta}}$. First, $u_{(j-1)(p'+1)+1}$ and $u_{(j-1)(p'+1)+1}$ are defined for $1\le j\le t$, and they satisfy $u_{\beta}\in V_{\overline{\beta}}$. If $u_{\beta}$ is a vertex not defined with the minimum index, from the definition of $\overline{\beta}$, we get $\overline{\beta}\ge 2$ and $v\not\in V_{\overline{\beta}}$ for any $v\in K_{\beta}$. Furthermore, $|K_{\beta}\cap W_{1,2}|\le 2$ if $\overline{\beta}\ge 3$, and $|K_{\beta}\cap W_{1,2}|\le 1$ if $\overline{\beta}=2$. For any $v\in K_{\beta}\setminus W_{1,2}$, we still get \eqref{ineq3}. For any $w\in  K_{\beta}\cap W_{1,2}$, let  $s=\min\{ d_{V_j}(w)\}$ for $3\leq j\leq p'$. Then, we have $d_{V_1}(w)\leq d_{V_2}(w)\leq s$ and
\[3s+(p'-3)(\frac{n}{p'}+3\sqrt{\varepsilon}n)\geq d_{V_1}(w)+d_{V_2}(w)+s+(p'-3)(\frac{n}{p'}+3\sqrt{\varepsilon}n)\geq d(w)\ge (1-\frac{1}{p'}-10\sqrt{\varepsilon})n.\]
This yields that, for $3\leq j\leq p'$
\begin{equation}\label{ineq4}
    d_{V_j}(w)\geq \left(\frac{2}{3p'}-\frac{4\sqrt{\varepsilon}}{3}p'\right)n.
\end{equation}
Hence, if $\overline{\beta}=2$, from Lemma \ref{clm-3-12} and \eqref{ineq3}, we get
\[\left|\bigcap_{v\in K_{\beta}}N_{V_{\overline{\beta}}}(v)\right|\ge  \left(\frac{1}{2p'}-3\sqrt{\varepsilon}p'\right)n+(|K_{\beta}|-1)\left( \frac{1}{p'} - \alpha \right)n-(|K_{\beta}|-1)\left(\frac{1}{p'}+3\sqrt{\varepsilon}\right)n>\alpha n;\]
if $\overline{\beta}\ge 3$, from  Lemma \ref{clm-3-12} and \eqref{ineq4}, we get
\[\left|\bigcap_{v\in K_{\beta}}N_{V_{\overline{\beta}}}(v)\right|\ge 2\left(\frac{2}{3p'}-\frac{4\sqrt{\varepsilon}}{3}p'\right)n+(|K_{\beta}|-2)\left( \frac{1}{p'} - \alpha \right)n-(|K_{\beta}|-1)\left(\frac{1}{p'}+3\sqrt{\varepsilon}\right)n>\alpha n.\]
Therefore, we could always choose $u_{\beta}$ to be an arbitrary vertex in $\left(\bigcap_{v\in K_{\beta}}N_{V_{\overline{\beta}}}(v)\right)\setminus W$. By repeating these processes, we get the desired vertices $\{u_1,\ldots,u_k\}$.\end{proof}

Let $W_i^* = \{ v \in W_i \mid d_{V_j}(v) \geq \left( \frac{1}{2p'} + \alpha \right) n \text{ for any }j\ne i \}$, and $W'_i=W_i\setminus W_i^*$. Denote by $W^* = \bigcup_{i=1}^{p'} W_i^*$ and $W'=\bigcup_{i=1}^{p'} W'_i$.

\begin{lemma}\label{lem-3-13}
    $|W^*|=|W|=t-1$. 
\end{lemma}
\begin{proof}
 Note that $|W|\ge t-1$ due to Lemma \ref{lem-3-11}. It only needs to prove $|W|\le t-1$. Suppose to the contrary that $|W| \geq t$. 
 
 First, we claim that $|W^*| \leq t-2$. Otherwise, there exist one vertex $v^{(1)}_{i_1,1}$ from $W \setminus W^*$ and $t - 1$ vertices $v^{(2)}_{i_2,1},\ldots,v^{(t)}_{i_t,1}$ from $W^*$ satisfying $v^{(j)}_{i_j,1}\in V_{i_j}$ for $1\le j\le t$. Since $d_{V_{i_j}}(v^{(j)}_{i_j,1})\ge 2\sqrt{\varepsilon}n$ and $|W|\le \sqrt{\varepsilon}n$, there exists a matching $M=\{v^{(1)}_{i_1,1}v^{(1)}_{i_1,2},v^{(2)}_{i_2,1}v^{(2)}_{i_2,2},\ldots,v^{(t)}_{i_t,1}v^{(t)}_{i_t,2}\}$ with $v^{(j)}_{i_j,2}\in V_{i_j}\setminus W$. By the same method as that in the proof of Lemma \ref{lem-3-10}, we could also define the vertices $\{u_1,u_2,\ldots,u_k\}$ that form a $C_k^p$, which leads to the contradiction. In detail, adopting the setting and notation from the proof of Lemma \ref{lem-3-10}, we claim that, if $u_{\beta}$ is defined, then $u_{\beta}\in V_{\overline{\beta}}$. If $u_{\beta}$ is the vertex not defined with the minimum index, from the definition, we get $v\not\in V_{\overline{\beta}}$ for any $v\in K_{\beta}$. Moreover, $|K_{\beta}\cap W|\le 2$ and $|K_{\beta}\cap W'|\le 1$. Therefore, from Lemma \ref{clm-3-12} and Eq.\eqref{ineq3}, we get
\[\begin{array}{lll}
\left|\bigcap_{v\in K_{\beta}}N_{V_{\overline{\beta}}}(v)\right|&\ge& \left(\frac{1}{2p'}-3\sqrt{\varepsilon}p'\right)n+\left(\frac{1}{2p'}+\alpha\right)n+(|K_{\beta}|-2)\left(\frac{1}{p'}-6\sqrt{\varepsilon}p'\right)n\\[2mm]
&&-(|K_{\beta}|-1)\left(\frac{1}{p'}+3\sqrt{\varepsilon}\right)n\\[2mm]
 &=&\left(\alpha-3\left((2|K_{\beta}|-3)p'+|K_{\beta}|-1\right)\sqrt{\varepsilon}\right)n>\sqrt{\varepsilon}n.
 \end{array}\]
 Therefore, we could always choose $u_{\beta}$ to be an arbitrary vertex in $\left(\bigcap_{v\in K_{\beta}}N_{V_{\overline{\beta}}}(v)\right)\setminus W$. By repeating these processes, we get the desired vertices $\{u_1,\ldots,u_k\}$.
 
Next, for $w\in W'$, assume that $w\in V_{i_w}$ and that $j_w$ is the unique index satisfying $d_{V_{j_w}}(w)\le \left(\frac{1}{2p'}+\alpha\right)n$. Taking an arbitrary vertex $v'\in V\setminus W$, define the graph $G'$ by setting  
\[G'=G-\bigcup_{i=1}^{p'}E(G[V_i\setminus W_i])-\bigcup_{w\in W'}E(G[w,N_{V_{i_w}}(w)])+\bigcup_{w\in W'}E(G[w,V_{j_w}\setminus N_{V_{j_w}}(w)])+E(G[v',V\setminus N(v')]).\]
Therefore, $V_{\ell}'=(V_{\ell}\setminus W_{\ell})\cup\{w\mid w\in W'\text{ and }i_w=\ell\}$ forms an independent set for $1\le \ell\le p'$. Let $n_{\ell}=|V_{\ell}'|$. Hence, we get
\[G'\subseteq G'[W^*\cup\{v'\}]\vee K_{n_1,\ldots,n_{p'}}\subseteq K_{t-1}\vee H,\] 
where $H$ is a complete $p'$-partite graph on $n-t+1$ vertices. Thus, $G'$ is $C_k^p$-free. In what follows, we shall obtain a contradiction by showing that $\lambda(G')>\lambda(G)$. Note that $d_G(v')\le (2\sqrt{\varepsilon})n+(p'-1)(1/p'+3\sqrt{\varepsilon})n=(1+3\sqrt{\varepsilon}p'-1/p'-\sqrt{\varepsilon})n$. Combining this with Lemmas \ref{lem-3-7} and \ref{lem-3-10}, we get
\begin{equation}\label{eq-5}\begin{array}{lll}
    &&\sum_{u\in V\setminus N_G(v')}x_ux_{v'}-\sum_{uv\in \bigcup_{i=1}^{p'} E(G[V_i \setminus W_i])}x_ux_v\\[2mm]
    &\geq& \sum_{u\in V\setminus (N_G(v')\cup W)}x_ux_{v'}-\sum_{uv\in \bigcup_{i=1}^{p'} E(G[V_i \setminus W_i])}x_ux_v\\[2mm]
    &\geq&  \left(1-(1+3\sqrt{\varepsilon}p'-1/p'-\sqrt{\varepsilon})-\sqrt{\varepsilon}\right)(1-\varepsilon_1)^2nx^2_{v^*}-4(t-1)\sqrt{\varepsilon}nx^2_{v^*}\\
    &=&\left(\left(\frac{1}{p'}-3p'\sqrt{\varepsilon}\right)(1-\varepsilon_1)^2-4(t-1)\sqrt{\varepsilon}\right)nx_{v^*}^2\ge (\frac{1}{p'}-\alpha)nx^2_{v^*}.
\end{array}
\end{equation}
Besides, for any $w\in W'$, by Lemmas \ref{clm-3-12} and \ref{lem-3-9}, we have
\begin{equation}\label{eq-6}
\begin{array}{lll}
  &&\sum_{u\in V_{j_w}\setminus(N_G(w)\cap V_{j_w})}x_wx_u-\sum_{v\in N_G(w)\cap V_{i_w}}x_wx_v\\[2mm]
  &=&x_w\left(\sum_{u\in V_{j_w}\setminus(N_G(w)\cap V_{j_w})}x_u-\sum_{v\in N_G(w)\cap V_{i_w}}x_v\right)\\[2mm]
  &\ge&x_w\left(\sum_{u\in V_{j_w}\setminus((N_G(w)\cap V_{j_w})\cup W)}x_u-\sum_{v\in N_G(w)\cap V_{i_w}}x_v\right)\\[2mm]
  &\ge&x_w\left(\sum_{u\in V_{j_w}\setminus((N_G(w)\cap V_{j_w})\cup W)}x_u-\sum_{v\in (N_G(w)\cap V_{i_w}\cap W)}x_v-\sum_{v\in (N_G(w)\cap V_{i_w})\setminus W}x_v\right)\\[2mm]
  &\ge&x_w\left(\left(\left(\frac{1}{p'}-3\sqrt{\varepsilon}\right)-\left(\frac{1}{2p'}+\alpha\right)-\sqrt{\varepsilon}\right)(1-\varepsilon_1)nx_{v^*}-|W|-\left(\frac{1}{2p'}+\alpha\right)nx_{v^*}\right)\\[2mm]
  &\ge&\left(\left(\frac{1}{2p'}-4\sqrt{\varepsilon}-\alpha\right)(1-\varepsilon_1)-3\sqrt{\varepsilon}-\left(\frac{1}{2p'}+\alpha\right)\right)nx_wx_{v^*}\\[2mm]
  &=&\left(-2\alpha-7\sqrt{\varepsilon}-\left(\frac{1}{2p'}-4\sqrt{\varepsilon}-\alpha\right)\varepsilon_1\right)nx_wx_{v^*}\\[2mm]
  &\ge& -3\alpha nx_wx_{v^*}\ge-3\alpha nx_{v^*},
  \end{array}
\end{equation}
where the third last inequality follows from $3x_{v^*}\ge 1$ and $|W|\le \sqrt{\varepsilon}n$.
Consequently, combining Eqs. \eqref{ineq2}, \eqref{eq-5}, \eqref{eq-6}, and Lemma \ref{lem-3-11}, we get
\[\begin{array}{lll}
     \lambda(G')-\lambda(G)&\geq& \frac{2}{xx^T}\left(\sum_{uv\in E(G')}x_ux_v-\sum_{uv\in E(G)}x_ux_v\right)\\[2mm]
    &\geq& \frac{2}{xx^T}\left(\left(\sum_{u\in V\setminus N_G(v')}x_ux_{v'}-\sum_{uv\in \bigcup_{i=1}^{p'} E(G[V_i \setminus W_i])}x_ux_v\right)\right.\\[2mm]
    &&\left.+\sum_{w\in W'}\left(\sum_{u\in V_{j_w}\setminus(N_G(w)\cap V_{j_w})}x_wx_u-\sum_{v\in N_G(w)\cap V_{i_w}}x_wx_v\right)\right)\\[2mm]
    &\ge&\frac{2}{xx^T}\left((\frac{1}{p'}-\alpha)nx^2_{v^*}+|W'|(-3\alpha nx_{v^*})\right)\\[2mm]
    &\ge&\frac{2}{xx^T}\left((\frac{1}{p'}-\alpha)nx^2_{v^*}-3\alpha p'(p'-1)t nx_{v^*}\right)\\[2mm]
    &\ge&\frac{2}{xx^T}\left((\frac{1}{p'}-\alpha)(1-\frac{5}{4p'})-3\alpha p'(p'-1)t \right)nx_{v^*}\\[2mm]
    &>&0,
 \end{array}\]
 which contradicts the maximality of $\lambda(G)$. Then, we get $|W|=t-1$. Furthermore, we conclude that $|W^*| = t-1$ since otherwise the graph $G'$ is still $C_k^p$-free satisfying $\lambda(G') > \lambda(G)$. 
 
 The proof is completed.
\end{proof}

\begin{lemma}\label{lem-a-1}
 $e(G[\bigcup_{i=1}^{p'}V_i\setminus W])=0$.
\end{lemma}
\begin{proof}
Conversely, assume that there exists an edge in $\bigcup_{i=1}^{p'}G[V_i\setminus W]$, say $v^{(1)}_{1,1}v^{(1)}_{1,2}\in E(G[V_1\setminus W])$. Therefore, there exist a matching $M=\{v^{(1)}_{1,1}v^{(1)}_{1,2},v^{(2)}_{i_2,1}v^{(2)}_{i_2,2},\ldots, v^{(t)}_{i_t,1}v^{(t)}_{i_t,2}\}$ satisfying $v^{(j)}_{i_j,1}\in W_{i_j}$ and $v^{(j)}_{i_j,2}\in V_{i_j}\setminus W_{i_j}$ for $2\le j\le t$. By the same method as that in the proof of Lemma \ref{lem-3-10}, we could also define the vertices $\{u_1,u_2,\ldots,u_k\}$ that form a $C_k^p$, which leads to the contradiction. In detail, adopting the setting and notation from the proof of Lemma \ref{lem-3-10}, we claim that, if $u_{\beta}$ is defined, then $u_{\beta}\in V_{\overline{\beta}}$. If $u_{\beta}$ is the vertex not defined with the minimum index, from the definition, we get $v\not\in V_{\overline{\beta}}$ for any $v\in K_{\beta}$, and $|K_{\beta}\cap W|\le 2$. Therefore, from Eq.\eqref{ineq3} and $W=W*$, we get
\[\begin{array}{lll}
\left|\bigcap_{v\in K_{\beta}}N_{V_{\overline{\beta}}}(v)\right|&\ge& 2\left(\frac{1}{2p'}+\alpha\right)n+(|K_{\beta}|-2)\left(\frac{1}{p'}-6\sqrt{\varepsilon}p'\right)n-(|K_{\beta}|-1)\left(\frac{1}{p'}+3\sqrt{\varepsilon}\right)n\\[2mm]
&=&\left(2\alpha-3(2(|K_{\beta}|-2)p'+|K_{\beta}|-1)\sqrt{\varepsilon}\right)n>\alpha n.\end{array}\]
Therefore, we could always choose $u_{\beta}$ to be an arbitrary vertex in $\left(\bigcap_{v\in K_{\beta}}N_{V_{\overline{\beta}}}(v)\right)\setminus W$. By repeating these processes, we get the desired vertices $\{u_1,\ldots,u_k\}$.  
\end{proof}

\begin{lemma}\label{lem-3-14}
    For any $u\in W$, $d(u)=n-1$.
\end{lemma}
\begin{proof}
    Suppose to the contrary that there exists a vertex $u \in W$ such that $d(u) < n - 1$. Let $v \in V(G)$ be a vertex with $uv \notin E(G)$, and define $G' = G + \{uv\}$. Since $G' \subseteq K_{t-1} \vee H$, where $H$ is a complete $p'$-partite graph on $n - t + 1$ vertices. Hence, $G'$ is $C_k^p$-free. However, it is clear that $\lambda(G') > \lambda(G)$, which contradicts the maximality of $\lambda(G)$.
    \end{proof}

\begin{proof}[{\bf Proof of the Theorem \ref{thm-2}.}] Now we prove that $G\cong K_{t-1}\vee T_{n-t+1,p'}$. For any $1\le i\le p'$, let $|V_i\setminus W|=n_i$. By Lemma \ref{lem-3-13} and Lemma \ref{lem-3-14}, there exists a $p'$-partite graph $H$ with classes of size $n_1,n_2,\ldots,n_{p'}$ such that $G\cong K_{t-1}\vee H$.  The maximality of $\lambda(G)$ yields $H\cong K_{n_1,n_2,\ldots,n_{p'}}$. It suffices to show that $|n_i-n_j|\leq 1$ for any $1\leq i<j\leq p'$. Assume that $n_1\geq n_2\geq \cdots\geq n_r$. Suppose to the contrary that $n_1-n_{p'}\geq 2$. Let $H'=K_{n_1-1,n_2,\ldots,n_{p'-1},n_{p'}+1}$, and $G'=K_{t-1}\vee H'$.

Recall that $x$ is the eigenvector of $G$ corresponding to $\lambda(G)$, by the symmetry we may assume $x=(x_1,\ldots,x_1,x_2,\ldots,x_2,\ldots,x_{p'},\ldots,x_{p'},x_{p'+1},\ldots,x_{p'+1})$, where the number of $x_i$ is $n_i$ for $1\leq i\leq p'$ and the number of $x_{p'+1}$ is $t-1$. Therefore, we get
$$\lambda(G)x_i=\sum_{j=1}^{p'}n_jx_j-n_ix_i+(t-1)x_{p'+1}, \text{ for $1\le i\le p'$}, $$ 
and 
$$\lambda(G)x_{p'+1}=\sum_{j=1}^{p'}n_jx_j+(t-2)x_{p'+1}.$$
It follows that $x_i=\frac{\lambda(G)+1}{\lambda(G)+n_i}x_{p'+1}$ for any $1\le i\le p'$, which implies that $x_{p'+1}=\max\{x_v\mid v\in V(G)\}=1$, and thus $x_i=\frac{\lambda(G)+1}{\lambda(G)+n_i}$ for $1\le i\le p'$. Let $u_0\in V_1\setminus W$ be a fixed vertex. Then $G'$ can be obtained from $G$ by deleting all edges between $u_0$ and $V_{p'}\setminus W$, and adding all edges between $u_0$ and $V_{1}\setminus (W\cup \{u_{0}\})$. Then, we have 
\begin{align*}
    \lambda(G')-\lambda(G)&\geq \frac{{x}^T (A(G')-A(G)){x}}{{x^T}{x}}\\
    &=\frac{2}{x^Tx}\left((n_1-1)x^2_{1}-n_{p'}x_{1}x_{p'}\right)\\
    &=\frac{2x_{1}}{x^Tx}\left((n_{1}-1)\frac{\lambda(G)+1}{\lambda(G)+n_{1}}-n_{p'}\frac{\lambda(G)+1}{\lambda(G)+n_{p'}}\right)\\
    &=\frac{2x_{1}}{x^Tx}\frac{(\lambda(G)+1)(n_{1}\lambda(G)-n_{p'}\lambda(G)-\lambda(G)-n_{p'})}{(\lambda(G)+n_{1})(\lambda(G)+n_{p'})}\\
    &\geq \frac{2x_{1}}{x^Tx}\frac{(\lambda(G)+1)(\lambda(G)-n_{p'})}{(\lambda(G)+n_{1})(\lambda(G)+n_{p'})}>0,
\end{align*}
where the last inequality holds since $\lambda(G)\geq (1-\frac{1}{p'})n$, and $n_{p'}\leq \frac{n}{p'}+3\sqrt{\varepsilon}n-(t-1)$. This contradicts the maximality of $\lambda(G)$.
\end{proof}

\section*{Acknowledgments}
This work is supported by NSFC (No. 12371362).

\section*{Declaration of competing interest}
We declare that we have no conflict of interest to this work.

\end{document}